\documentclass[12pt,oneside]{amsart}

\usepackage[english]{babel}
\usepackage{amsmath,amssymb,amsfonts}

\RequirePackage{dsfont} 

\usepackage{geometry}
\newgeometry{tmargin=3cm, bmargin=3cm, lmargin=2cm, rmargin=2cm}

 \scrollmode

\usepackage{cite}

\theoremstyle{plain}
\newtheorem{theorem}{Theorem}
\newtheorem{lemma}{Lemma}
\newtheorem{corollary}{Corollary}
\newtheorem{proposition}{Proposition}

\theoremstyle{definition}
\newtheorem{remark}{Remark}
\newtheorem{example}{Example}

\def\esett{\mathrm e^\mathrm{sett}}
\def\csett{\mathrm{cost}^\mathrm{sett}}
\def\nsett{\mathrm{comp}^\mathrm{sett}}

\def\ew{\mathrm e^\mathrm{ww}}
\def\rw{\mathrm{rad}^\mathrm{ww}}
\def\cw{\mathrm{cost}^\mathrm{ww}}
\def\nw{\mathrm{comp}^\mathrm{ww}}
\def\nsw{\mathrm{n}^\mathrm{w}}

\def\ea{\mathrm e^\mathrm{wa}}
\def\ca{\mathrm{cost}^\mathrm{wa}}
\def\na{\mathrm{comp}^\mathrm{wa}}

\def\by{\mathbf y}

\title
[Tractability in the presence of noise]
{Worst case tractability of linear problems\\in the presence of noise: linear information}

\author[L. Plaskota]{Leszek Plaskota}
\address{Faculty of Mathematics, Informatics and Mechanics, University of Warsaw, 
ul. S.~Banacha 2, 02-097 Warsaw, Poland}
\email{L.Plaskota@mimuw.edu.pl, https://orcid.org/0000-0001-8704-0790}

\author[P. Siedlecki]{Pawe{\l} Siedlecki}
\address{Faculty of Mathematics, Informatics and Mechanics, University of Warsaw, 
ul. S.~Banacha 2, 02-097 Warsaw, Poland}
\email{P.Siedlecki@mimuw.edu.pl, https://orcid.org/0000-0001-7352-9253}

\date{\today}

\begin{document}

\begin{abstract}
We study the worst case tractability of multivariate linear problems defined on separable Hilbert spaces. 
Information about a problem instance consists of noisy evaluations of arbitrary bounded linear functionals, 
where the noise is either deterministic or random. The cost of a single evaluation depends on its precision 
and is controlled by a cost function. We establish mutual interactions between tractability of a problem with 
noisy information, the cost function, and tractability of the same problem, but with exact information.
\end{abstract}

\maketitle

\section{Introduction}

\emph{Tractability of multivariate problems} is nowadays one of the most active areas 
of \emph{information-based complexity}; 
we mention only the three-volume monograph \cite{NoWo08,NoWo10,NoWo12}. 
Tractability research concentrates on establishing both quantitative and qualitative properties of 
the interplay between the cost and accuracy of approximation, and the number of variables 
occurring in a multivariate computational problem. To the best of the authors' knowledge, all 
tractability research has hitherto concentrated on \emph{exact information}, i.e., information consisting 
of exact evaluations of information functionals. The goal of this article is to extend tractability 
studies to include \emph{noisy information}, where observations of functionals 
are contaminated by some noise. 

We study tractability in the \emph{worst case setting}, in the presence of 
\emph{deterministic} (bounded) or \emph{random} (Gaussian) noise. The model of noise and cost 
is adopted from \cite{MoPl20,NICC96}. That is, information is built out of a finite number of 
noisy evaluations of functionals, which 
are subject to our choice. Moreover, prior to their noisy evaluation it is 
also possible to set required \emph{precision} $\sigma,$ which is a bound on the absolute 
value of the noise in the deterministic case, and the standard deviation of a Gaussian variable 
in the case of random noise. The cost of a single evaluation with a given precision is controlled by 
a \emph{cost function} $\$$, which is a part of the problem formulation. 
The higher the precision, the higher the cost.

The main theme of our work is a comparative study of exact and noisy information from the point 
of view of tractability of multivariate linear problems $S_d:F_d\to G_d$ acting between separable Hilbert spaces.
We assume that noisy evaluations of \emph{any} linear functionals with norm bounded by one are possible.
The focus is on \emph{(strong) polynomial tractability}, \emph{weak tractability}, \emph{intractability}, and 
\emph{the curse of dimensionality}. We are interested in establishing mutual interactions between tractability 
of a multivariate problem with noisy information, the cost function, and tractability of the same problem, 
but with exact information. In particular, we seek for conditions guaranteeing equivalence of 
various tractability notions for both, the exact and noisy settings. 

Such equivalence is established, for instance, for polynomial tractability provided the cost function 
grows polynomially.  
To give a flavor of our results, suppose that the problem with exact information is polynomially 
tractable, i.e., its $(\varepsilon,d)$-complexity is upper bounded by $Cd^q\varepsilon^{-p},$ 
where $\varepsilon$ is the required error of approximation, 
and that the cost function grows polynomially, i.e., $\$(\sigma,d)\le 1+Dd^t\sigma^{-2s}.$ 
Then the same problem with noisy information is also polynomially tractable. 
Moreover, its complexity is essentially bounded as
$$\mathrm{comp}_\$(\varepsilon,d)\preccurlyeq 
d^{\overline t+q(\overline s+1)}
\varepsilon^{-\max\left(p(\overline s+1),2\overline s\right)},$$
where $(\overline s,\overline t)=(s,t)$ for bounded noise, and 
$(\overline s,\overline t)=(s,t)/\max(1,s)$ for Gaussian noise, 
see Theorem~\ref{thm:wwpoly} and Theorem~\ref{thm:wapoly}. 
We stress that we do not know whether 
the exponents of polynomial tractability above are optimal. The point is that, unlike in the case of 
exact information, it is generally an open question how to optimally select functionals when their 
evaluations are corrupted by noise.

As for the technical part, it turns out that an important role in the analysis plays the complexity of 
a one-dimensional problem that relies on approximating an unknown real parameter from its noisy 
observations. This problem is trivial in the case of bounded noise, but far from that in the case 
of Gaussian noise, cf. \cite{Don94,MoPl20}. Some difficulty in showing lower bounds adds the fact 
that in the case of random noise one has to consider deterministic as well as randomized approximations. 
Indeed, although randomization is formally not allowed in the problem formulation, it can be mimicked 
with the help of adaption, cf. \cite{Pla96a,Pla96b}.

The paper is organized as follows. 
The scene is formally set in Section~\ref{sec:prelim}. The results for bounded noise 
are in Section~\ref{sec:determ}, and those for Gaussian noise in Section~\ref{sec:random}. 
The Appendix contains some additional material concerning the optimal choice of information 
functionals in the case of bounded and Gaussian noise.

\section{Preliminaries}\label{sec:prelim}

We consider a \emph{multivariate problem} $\mathcal S=\{S_d\}_{d\ge 1}$ where
$$  S_d: F_d\to G_d, $$
$F_d$ and $G_d$ are separable Hilbert spaces, both over the reals, and 
$S_d$ are nonzero continuous linear operators with norms 
$$  \|S_d\|=\sup_{\|f\|_{F_d}\le 1}\|S_d(f)\|_{G_d}. $$

\subsection{Information and approximation}
The values $S_d(f)$ for $f\in F_d$ are approximated based on information 
$\mathbf y=(y_1,y_2,\ldots,y_n)\in\mathbb R^n$ about $f,$ which consists of finitely many 
noisy values of some functionals at $f.$ That is, 
$$ y_i=L_i(f)+e_i,\quad 1\le i\le n, $$ 
where $L_i$ are in a class $\Lambda_d\subset F_d^*$ of permissible functionals,  and 
$e_i$ is noise. A crucial assumption of the current paper is that  arbitrary continuous functionals 
with norm at most one are allowed, 
$$\Lambda_d=\{L\in F_d^*:\,\|L\|\le 1\},$$ where $\|L\|=\sup_{\|f\|_{F_d}\le 1}|L(f)|.$
The noise can be deterministic (bounded) or random (Gaussian), 
$$  |e_i|\le\sigma_i\quad\mbox{or}\quad e_i\stackrel{iid}\sim\mathcal N(0,\sigma_i^2), $$
where $\sigma_i$ represents precision of the $i$th evaluation, and $\mathcal N(0,\sigma)$ 
is the standard zero-mean Gaussian distribution with variance $\sigma^2.$ 
Then an approximation to $S_d(f)$ is given as $\Phi(\mathbf y),$ where 
$$\Phi:Y\to G_d,$$
called an \emph{algorithm}, is an arbitrary mapping 
acting on the set $Y$ of all possible values of information.

\smallskip
We now describe the information more formally. We first deal with \emph{nonadaptive} (or parallel)
information, in which case the functionals $L_i$ and precisions $\sigma_i$ are the same for all problem 
instances $f\in F_d.$ In the case of bounded noise, nonadaptive information is a multi-valued 
operator, i.e., $N:F_d\to 2^Y,$ where $2^Y$ is the power set of $Y=\mathbb R^n,$ and
$$N(f)=\big\{\big(L_1(f)+e_1,L_2(f)+e_2,\ldots,L_n(f)+e_n\big):\;
     |e_i|\le\sigma_i,\,1\le i\le n\big\}.$$
Then $\mathbf y$ is information about $f$ iff $\mathbf y\in N(f).$ 

In case of Gaussian noise, nonadaptive information $\mathbf y$ about $f$ is a realization of 
the random variable with $n$ dimensional Gaussian distribution $\pi_f$ whose mean element is
$m_f=(L_1(f),\ldots,L_n(f))$ and correlation matrix 
$\Sigma=\mathrm{diag}(\sigma_1^2,\ldots,\sigma_n^2).$ 
Therefore nonadaptive information is now a mapping $N:F_d\to\mathcal P(Y),$ where 
$\mathcal P(Y)$ is a set of probability distributions on the Borel sets of $Y=\mathbb R^n,$ and 
$$N(f)=\pi_f\quad\mbox{for}\quad f\in F_d.$$

Although we will mainly exploit nonadaptive information in this paper, in a generic approximation 
scheme we also allow a more general \emph{adaptive} (or sequential) information, where 
the choice of the successive functionals $L_i$ and precisions $\sigma_i,$ as well as the number 
of them, depend on $f$ and noise via the previously obtained values $y_1,\ldots,y_{i-1}.$ 
The process of obtaining adaptive information $\mathbf y=(y_1,\ldots,y_n)$ about $f$ can be 
schematically described as follows:
\begin{eqnarray}\label{adaptinfo}\left\{\;\begin{array}{llll}
   y_1 &=& L_1(f)+e_1, & \quad\sigma_1, \\
    y_2 &=& L_2(f;y_1)+e_2,  & \quad\sigma_2(y_1), \\ 
    y_3 &=& L_3(f;y_1,y_2)+e_3, & \quad\sigma_3(y_1,y_2), \\
    &\cdots &  & \\
   y_n &=& L_n(f;y_1,y_2,\ldots,y_{n-1})+e_n, & \quad\sigma_n(y_1,y_2,\ldots,y_{n-1}),
    \end{array}\right.
\end{eqnarray}
where $L_i(\,\cdot\,;y_1,\ldots,y_{i-1})\in\Lambda_d.$
The process terminates when $(y_1,y_2,\ldots,y_n)\in Y,$ where the set $Y$ of all values 
of information consists of finite sequences of (possibly) various lengths. For the termination criterion 
to be well defined we assume that for any infinite sequence $(y_1,y_2,y_3\ldots)$ there is exactly 
one $n$ such that $(y_1,\ldots,y_n)\in Y.$ The corresponding operator $N$ is for both, bounded 
and Gaussian noise, determined  by the above construction. (In case of Gaussian noise appropriate 
measurability assumptions on $L(f;\,\cdot)$ and $\sigma_i(\cdot)$ have to be met.) 
For details, see \cite[Sect.~2.7 \&~3.7]{NICC96}.

\subsection{Cost function}
We assume that we are free to choose the information functionals and precisions, but we have to 
pay more for more accurate evaluations. That is, the cost of a single noisy evaluation of $L(f)$ 
for $f\in F_d$ with precision $\sigma$ equals $\$(\sigma,d),$ where 
$$ \$:[0,+\infty)\times\{1,2,3,\ldots\}\to[1,+\infty] $$
is a \emph{cost function} that is non-decreasing in both $\sigma^{-1}$ and $d.$ 
Note that $\$\ge 1,$ which corresponds to a natural assumption that one has to pay at least one unit
even for `slightest touch' of a functional. For instance,
$$  \$(\sigma,d)=\left\{\begin{array}{rl} 
      +\infty, & 0\le\sigma<\sigma_0, \\ 1, & \sigma_0\le\sigma, \end{array}\right. 
$$
corresponds to the situation when one can only observe with precision $\sigma_0$ at cost $1$. 
If, in addition, $\sigma_0=0$ then information is exact at the unit cost for all 
$\sigma\ge 0$ and $d\ge 1.$
We distinguish several types of cost functions depending on how they grow as 
$\sigma^{-1}$ and $d$ increase. In particular, we have:
\begin{itemize}
\item polynomial growth in $\sigma^{-1}$ and $d$ iff
$$\$(\sigma,d)\le 1+Dd^t\sigma^{-s}
\quad\mbox{for all}\;\,d\ge 1\;\mbox{and}\;\sigma\in(0,1),$$
where $D,t,s$ are some nonnegative numbers,
\item sub-exponential growth in $\sigma^{-1}+d$ iff
$$\lim_{\sigma^{-1}+d\to\infty}\frac{\ln\$(\sigma,d)}{\sigma^{-1}+d}=0,$$
\item exponential growth in $\sigma^{-1}+d$ iff
$$\limsup_{\sigma^{-1}+d\to\infty}\frac{\ln\$(\sigma,d)}{\sigma^{-1}+d}>0.$$
\end{itemize}
We will also consider corresponding growths in only one of the variables, $\sigma^{-1}$ or $d,$
with the other variable fixed. For instance, we have polynomial growth in $\sigma^{-1}$ iff 
$\$(\sigma,d)\le D\psi(d)\sigma^{-s}$ for all $d\ge 1$ and $\sigma\in(0,1),$ or we have 
sub-exponential growth in $d$ iff $\lim_{d\to\infty}\ln\$(\sigma,d)/d=0$ for all $\sigma\in(0,1).$

\medskip
A total cost $\csett_\$(N)$ of given information $N$ and error $\esett(S_d,N,\Phi)$ of 
an algorithm $\Phi$ using it depend on a setting under consideration, and will be defined 
separately for each setting. The settings ware distinguished by whether we have 
bounded or Gaussian noise.

\subsection{Tractability notions}
For a given setting, let 
$$  \nsett_\$(\varepsilon,d)=
      \inf\big\{\csett_\$(N):\,N,\Phi\;\mbox{such that}\;\esett(S_d,N,\Phi)\le\varepsilon\|S_d\|\,\big\}  $$
be the minimal cost of information sufficient to approximate $S_d$ with (normalized) error $\varepsilon.$
We call $\nsett_\$(\varepsilon,d)$ the \emph{information $(\varepsilon,d)$-complexity}, or simply 
\emph{$(\varepsilon,d)$-complexity} of our problem. We consider the following tractability notions,
cf. \cite{NoWo08}.
\begin{itemize}
\item
A multivariate problem $\mathcal S=\{S_d\}_{d\ge 1}$ is \emph{polynomially tractable} iff
\begin{equation}\label{polytract}
  \nsett_\$(\varepsilon,d)\le C d^q \varepsilon^{-p}
  \quad\mbox{for all}\;\,d\ge 1\;\mbox{and}\; \varepsilon\in(0,1),
\end{equation}
where $C,q,p$ are some nonnegative numbers. 
If, in addition, \eqref{polytract} holds with $q=0$ then the problem is
\emph{strongly polynomially tractable}, and the infimum of $p$ satisfying \eqref{polytract} 
with $q=0$ is the \emph{strong exponent}.
\item
A problem is \emph{weakly tractable} iff
$$  \lim_{\varepsilon^{-1}+d\to+\infty}\frac{\ln\left(\nsett_\$(\varepsilon,d)\right)}
{\varepsilon^{-1}+d}=0. $$
\item
A problem is \emph{intractable} iff it is not weakly tractable. 
\item
A problem suffers from the \emph{curse of dimensionality} iff 
there are $\varepsilon_0>0,$ $C>0,$ and $\gamma>0,$ such that for infinitely many $d$ we have
$$\nsett_\$(\varepsilon_0,d)\ge C(1+\gamma)^d.$$ 
Equivalently, we have the curse iff there is $\varepsilon_0>0$ such that
$$\limsup_{d\to\infty}\frac{\ln\big(\nsett_\$(\varepsilon_0,d)\big)}{d}>0.$$
\end{itemize}

We will later draw conclusions about tractability in the case of noisy information assuming we know
tractability for exact information. As we already noticed, in the latter case we have $\$(\sigma,d)=1,$ 
which means that we just count the number of functional evaluations. In the two settings considered 
in this paper the complexities in the case of exact information are the same and denoted by 
$$\nsw(\varepsilon,d),$$
where `$\mathrm w$' stands for `worst'. 

\section{Worst case setting with bounded noise}\label{sec:determ}

In this section we assume that the noise is bounded.
That is, information about $f$ is given as $\by=(y_1,y_2,\ldots,y_{n(\by)})$ where
$$y_i=L_i(f;y_1,\ldots,y_{i-1})+e_i,\qquad|e_i|\le\sigma_i(y_1,\ldots,y_{i-1}).$$ 
The (total) cost of information $N$ is defined as
$$  \cw_\$(N)=\sup_{\|f\|_{F_d}\le 1}\,\sup_{\by\in N(f)}
      \sum_{i=1}^{n(\by)}\$(\sigma_i(y_1,\ldots,y_{i-1})),  $$
and the error of an algorithm $\Phi$ using information $N$ as
$$  \ew(S_d,N,\Phi)=\sup_{\|f\|_{F_d}\le 1}\,\sup_{\by\in N(f)}\|S_d(f)-\Phi(\by)\|_{G_d}. $$
We assume that $S_d$ is a \emph{compact} operator which, as well known, is necessary if 
we want to assure that $\nw_\$(\varepsilon,d)<+\infty$ for all $\varepsilon>0.$ 

\smallskip
We now recall some auxiliary facts about the current setting that can be found, e.g., in \cite{NICC96}. 

Let $N:F_d\to 2^Y$ be arbitrary information and $\rw(N)$ be its \emph{radius}, i.e., 
the minimal error that can achieved using $N.$ If $N$ is nonadaptive and uses $n$ functionals $L_i$ 
with precisions $\sigma_i$ then
\begin{equation}\label{radworst}
\rw(N)=\max\big\{\|S_d(h)\|_{G_d}:\;\|h\|_{F_d}\le 1,\,|L_i(h)|\le\sigma_i,\,1\le i\le n\big\}.
\end{equation}
Next we notice that we can restrict our considerations to algorithms using nonadaptive information. 
Indeed, for any adaptive information $N^{\mathrm{ada}}$ of the form \eqref{adaptinfo} and with
range $Y^\mathrm{ada}$ one can define nonadaptive information $N^{\mathrm{non}}$ with 
range $Y^\mathrm{non}=\mathbb R^n$ where $n$ is such that 
$(\underbrace{0,\ldots,0}_n)\in Y^\mathrm{ada}$ and 
$$(y_1,\ldots,y_n)\in N^\mathrm{non}(f)\quad\mbox{ iff }\quad
y_i=L_i(f;\underbrace{0,\ldots,0}_{i-1})+e_i,\;|e_i|\le\sigma_i(\underbrace{0,\ldots,0}_{i-1}),
\quad 1\le i\le n.$$
Then $\rw(N^\mathrm{non})\le\rw(N^\mathrm{ada})$ 
and $\cw_\$(N^\mathrm{non})\le\cw_\$(N^\mathrm{ada}),$ which means that adaption does not help.
This and \eqref{radworst} imply that
$$\nw_\$(\varepsilon,d)=
\inf\big\{\cw_\$(N):\,N\mbox{-nonadaptive},\;\rw(N)\le\varepsilon\|S_d\|\big\}.$$

\medskip
To avoid notational difficulties, from now on we assume that $\mathrm{dim}(F_d)=+\infty,$ 
which can obviously be done without loss of generality.
Let $\{f_{d,j}^*\}_{j\ge 1}$ be the complete orthonormal system of eigenelements of 
$S_d^*S_d:F_d\to F_d,$ and  
$$\lambda_{d,1}\ge\lambda_{d,2}\ge\cdots\ge\lambda_{d,j}\ge\cdots$$
the corresponding eigenvalues. We have $\|S_d\|=\sqrt{\lambda_{d,1}}$ and 
$\lim_{j\to\infty}\lambda_{d,j}=0.$ Furthermore, in the noiseless case, information
\begin{equation}\label{optinfo}
N_n^d=\big(\langle\,\cdot\,,f_{d,1}^*\rangle_{F_d},\ldots,\langle\,\cdot\,,f_{d,n}^*\rangle_{F_d}\big)
\end{equation}
is $n$th optimal, and its radius $\mathrm{rad}^\mathrm w(N_n^d)=\sqrt{\lambda_{d,n+1}},$ cf. \cite{NoWo08}.
Hence
\begin{equation}\label{minn}
\nsw(\varepsilon,d)=\min\big\{n:\,\sqrt{\lambda_{d,n+1}}\le\varepsilon\sqrt{\lambda_{d,1}}\,\big\}.
\end{equation}

\smallskip
We first show a general though important result that will be used later. 
\begin{lemma}\label{simplelemma}
For all $\varepsilon\in(0,1)$ and $d\ge 1$ we have 
$$\nw_\$(\varepsilon,d)\ge\sum_{k=1}^{\nsw(\varepsilon,d)}
\$\left(\varepsilon\sqrt{\tfrac{\lambda_{d,1}}{\lambda_{d,k}}},d\right).$$
Hence, $\nw_\$(\varepsilon,d)\ge\max\big\{\nsw(\varepsilon,d)\$(1,d),\,\$(\varepsilon,d)\big\}.$
\end{lemma} 

\begin{proof}
Let $N$ be nonadaptive information using $m$ functionals $L_i$ with precisions $\sigma_i,$
such that $\rw(N)\le\varepsilon\|S_d\|.$ Assume without loss of generality that 
$\sigma_1\le\sigma_2\le\cdots\le\sigma_m.$ To prove the lemma, it suffices to show that
$$\sigma_k\le\varepsilon\sqrt{\tfrac{\lambda_{d,1}}{\lambda_{d,k}}}<1,
\qquad 1\le k\le\nsw(\varepsilon,d).$$

Let $k$ be as above. The inequality `$<$' follows from \eqref{minn}. 
To show `$\le$', define the linear subspace
$$V_{k-1}=\big\{f\in F:\;L_i(f)=0,\;1\le i\le k-1\,\big\}
\qquad(\mbox{where}\;V_{k-1}=F\;\mbox{if}\;k=1).$$
Obviously $\mathrm{codim}(V_{k-1})\le k-1.$
Since for any $h$ with $\|h\|_{F_d}\le\sigma$ is $|L_i(h)|\le\sigma,$ we have
\begin{eqnarray*}
{\rw(N_n)}
&\ge&\max\{\|S_d(h)\|_{F_d}:\;\|h\|_{F_d}\le 1,\,h\in V_{k-1},\,|L_i(h)|\le\sigma_i,\,k\le i\le m\}\\
&\ge&\max\{\|S_d(h)\|_{F_d}:\;\|h\|_{F_d}\le\sigma_k,\,h\in V_{k-1}\}\,\ge\,
\sigma_k\sqrt{\lambda_{d,k}},
\end{eqnarray*}
where we used the fact that the norm of $S_d$ restricted to the subspace $V_{k-1}$ 
is at least $\sqrt{\lambda_k}.$ Hence 
$\sigma_k\le\varepsilon\sqrt{\tfrac{\lambda_{d,1}}{\lambda_{d,k}}}$
since otherwise we would have 
$\rw(N)>\varepsilon\sqrt{\lambda_{d,1}}=\varepsilon\|S_d\|.$
\end{proof}

To achieve upper bounds on tractability, we will use noisy version of the nonadaptive information 
$N_n^d$ defined in \eqref{optinfo}. That is, for given $d,n$ and $\sigma_i$ we have 
$\by=(y_1,\ldots,y_n)\in N^d_n(f)$ iff
\begin{equation}\label{noisyinf}
y_i=\langle f,f_{d,i}^*\rangle_{F_d}+e_i,\quad\mbox{where}\quad |e_i|\le\sigma_i.
\end{equation}
The radius of $N_n^d$ can be estimated from above by the error of the approximation
$$\Phi^d_n(\by)=\sum_{i=1}^n y_i S_d(f_{d,i}^*).$$
Specifically, using $f=\sum_{i=1}^{\infty}\langle f,f_{d,i}^*\rangle_{F_d}f_{d,i}^*$ and orthogonality 
of $\{S_d(f_{d,i}^*)\}_{i\ge 1}$ in $G_d$ we have
\begin{eqnarray*}  \|S_d(f)-\Phi^d_n(\by)\|_{G_d}^2 &=&
   \bigg\|-\sum_{i=1}^n e_i S_d(f_{d,i}^*)+\sum_{i=n+1}^{+\infty}
    \langle f,f_{d,i}\rangle_{F_d}S_d(f_{d,i}^*)\bigg\|_{G_d}^2 \\
   &=& \sum_{i=1}^n\lambda_{d,i}|e_i|^2+
             \sum_{i=n+1}^{+\infty}\lambda_{d,i}\big|\langle f,f_{d,i}^*\rangle_{F_d}\big|^2.
\end{eqnarray*}
Taking the suprema with respect to $\|f\|_{F_d}\le 1$ and $|e_i|\le\sigma_i$ we obtain
\begin{equation}\label{error1}  \ew(N^d_n,\Phi^d_n)=
      \sqrt{\sum_{i=1}^n\sigma_i^2\lambda_{d,i}+\lambda_{d,n+1}}.  
\end{equation}
In particular, for exact information we restore the known result that 
$\ew(S_d,N^d_n,\Phi^d_n)=\sqrt{\lambda_{d,n+1}},$ which is the minimal error
when $n$ exact functional evaluations are used. The cost of such approximation is obviously
$\sum_{i=1}^n\$(\sigma_i,d).$
 
\subsection{Polynomial tractability}
We use the following asymptotic notation. For two nonnegative functions of 
$\varepsilon$ and $d$ we write 
$$\psi_1(\varepsilon,d)\preccurlyeq\psi_2(\varepsilon,d)\qquad\mbox{iff}\qquad 
{\psi_1(\varepsilon,d)}\le A\,{\psi_2(\varepsilon,d)},$$
for some $A<+\infty$ and all $\varepsilon\in(0,1)$ and $d\ge 1.$

\begin{theorem}\label{thm:wwpoly}
Consider a multivariate problem $\mathcal S=\{S_d\}_{d\ge 1}.$ 
\begin{itemize}
\item[(i)] 
The problem with noisy information is polynomially tractable if and only if 
\begin{itemize}
\item[$\bullet$] it is polynomially tractable for exact information, and
\item[$\bullet$] the cost function grows polynomially in $\sigma^{-1}$ and $d.$
\end{itemize} \smallskip
\item[(ii)]
The problem with noisy information is strongly polynomially tractable if and only if
\begin{itemize}
\item[$\bullet$] it is strongly polynomially tractable for exact information, and
\item[$\bullet$] the cost function grows polynomially in 
$\sigma^{-1}$ and is bounded in $d$ for any $\sigma>0.$ 
\end{itemize} \smallskip
\item[(iii)]
Suppose that $\nsw(\varepsilon,d)\le Cd^q\varepsilon^{-p}$ and 
$\$(\sigma,d)\le 1+Dd^t\sigma^{-2s}.$ \\ If $p=0$ and $s=0$ then 
$\nw_\$(\varepsilon,d)\preccurlyeq d^{t+q};$ otherwise
$$  \nw_\$(\varepsilon,d)\,\preccurlyeq\, d^{t+q(s+1)}
             \left\{\begin{array}{rl} 
               \varepsilon^{-p(s+1)}, &\; p(s+1)>2s, \\
               \ln^{s+1}(1/\varepsilon)\,\varepsilon^{-2s}, &\; p(s+1)=2s,\\
               \varepsilon^{-2s}, &\; p(s+1)<2s.
           \end{array}\right.
$$
\end{itemize}
\end{theorem}

\begin{proof}
Suppose that the problem is polynomially tractable for noisy information, i.e., 
$$\nw(\varepsilon,d)\le Cd^q\varepsilon^{-p}.$$ Then we have by Lemma \ref{simplelemma} 
that, on one hand, 
$$\nsw(\varepsilon,d)\le\nw(\varepsilon,d)\le Cd^q\varepsilon^{-p}$$ 
and, on the other hand, $$\$(\sigma,d)\le\nw(\sigma,d)\le 1+Cd^q\sigma^{-p}.$$
This proves the necessary conditions in (i) and (ii). The sufficient conditions follow from (iii). 

\smallskip
In the proof of (iii) we distinguish several cases. 

If $s=0$ and $p\ge 0$ then exact observations are possible at cost $1+Dd^t,$ and therefore 
$$\nw_\$(\varepsilon,d)\le (1+Dd^t)\nsw(\varepsilon,d)\le (1+D)Cd^{t+q}\varepsilon^{-p}
\preccurlyeq d^{t+q}\varepsilon^{-p}.$$ 

Assume $s>0.$ We first optimize the cost of obtaining an $\varepsilon$-approximation using 
information $N_n^d$ with precisions $\sigma_i\le 1$ together with the algorithm $\Phi_n^d.$
Let $\ew(S_d,N^d_n,\Phi^d_n)\le\varepsilon\sqrt{\lambda_1}.$
The cost of $N_n^d$ is upper bounded by
$$ \psi_n(\sigma_1,\ldots,\sigma_n)=n+Dd^t\sum_{i=1}^n\sigma_i^{-2s}.$$
Minimizing $\psi_n$ with respect to the condition
$\ew(S_d,N^d_n,\Phi^d_n)^2=\sum_{i=1}^n\sigma_i^2\lambda_{d,i}+\lambda_{d,n+1}
\le\lambda_{d,1}\varepsilon^2$ 
we obtain the optimal values
$$\hat\sigma_k^{-2}=
\bigg(\frac{\lambda_{d,k}}{\lambda_{d,1}}\bigg)^\frac{1}{s+1}
      \sum_{i=1}^n\bigg(\frac{\lambda_{d,i}}{\lambda_{d,1}}\bigg)^\frac{s}{s+1}
      \bigg(\varepsilon^2-\frac{\lambda_{d,n+1}}{\lambda_{d,1}}\bigg)^{-1},\qquad 1\le k\le n,$$ and
$$\psi_n(\hat\sigma_1,\ldots,\hat\sigma_n)=n+Dd^t  
    \left(\sum_{i=1}^n\bigg(\frac{\lambda_{d,i}}{\lambda_{d,1}}\bigg)^\frac{s}{s+1}\right)^{s+1}
    \bigg(\varepsilon^2-\frac{\lambda_{d,n+1}}{\lambda_{d,1}}\bigg)^{-s}.$$
Now, taking $n=\max\big(2,\nsw(\varepsilon/\sqrt2,d)\big)$ we have 
$\frac{\lambda_{d,n+1}}{\lambda_{d,1}}\le\frac{\varepsilon^2}{2}<\frac{\lambda_{d,n}}{\lambda_{d,1}}$ and
$$\hat\sigma_k^{-2}\ge\hat\sigma_n^{-2}=\bigg(\frac{\lambda_{d,n}}{\lambda_{d,1}}\bigg)^\frac{1}{s+1}
      \sum_{i=1}^n\bigg(\frac{\lambda_{d,i}}{\lambda_{d,1}}\bigg)^\frac{s}{s+1}
      \bigg(\varepsilon^2-\frac{\lambda_{d,n+1}}{\lambda_{d,1}}\bigg)^{-1}\ge 
      n\bigg(\frac{\lambda_{d,n}}{\lambda_{d,1}}\bigg)\varepsilon^{-2}>\frac n2\ge 1,$$
i.e., $0<\sigma_1\le\cdots\le\sigma_n<1.$
Then an $\varepsilon$-approximation is obtained at cost
\begin{equation}\label{compform}
  \cw_\$(N_n^d)\,\le\,n+2^sDd^t\left(\sum_{i=1}^n
  \left(\frac{\lambda_{d,i}}{\lambda_{d,1}}\right)^\frac{s}{s+1}\right)^{s+1}\varepsilon^{-2s}.
\end{equation}

\medskip
Assume now that we have polynomial tractability for exact information, i.e., 
$$ \nsw(\varepsilon,d)\le C d^q\varepsilon^{-p}\quad\mbox{for}\;
        d\ge 1\;\mbox{and}\;\varepsilon\in(0,1). $$ 
        
If $p=0$ then $\lambda_j=0$ for $j\ge\lfloor Cd^q\rfloor+1$ 
and we have from \eqref{compform} that
$$\cw_\$(N_n^d)\,\preccurlyeq\,d^t\big\lfloor Cd^q\big\rfloor^{s+1}\varepsilon^{-2s}\,
\preccurlyeq\,d^{t+q(s+1)}\varepsilon^{-2s}.$$

Assume $p>0.$ We need to estimate the ratios $\lambda_{d,j}/\lambda_{d,1}.$  
For $1\le j\le\lfloor Cd^q\rfloor+1$ we have $\lambda_{d,j}/\lambda_{d,1}\le 1.$ 
Let $j\ge\lfloor Cd^q\rfloor+2.$ Let $\varepsilon_j\in(0,1)$ be such that $j=Cd^q\varepsilon_j^{-p}+1.$ 
Then $j-1\ge\nsw(\varepsilon_j,d),$ which implies
$$\sqrt\frac{\lambda_{d,j}}{\lambda_{d,1}}\le\varepsilon_j=\left(\frac{Cd^q}{j-1}\right)^{1/p}.$$
Hence for all $j\ge 1$
\begin{equation}\label{lambdasb}
\frac{\lambda_{d,j}}{\lambda_{d,1}}\le\min\left(1,\Big(\frac{Cd^q}{j-1}\Big)^{2/p}\right).
\end{equation}

Assuming $C\ge 2$ (which can be done without loss of generality) 
the estimates \eqref{compform} and \eqref{lambdasb} give
$$ \nw_\$(\varepsilon,d)\le \left\lfloor Cd^q\big(\tfrac{\varepsilon}{\sqrt2}\big)^{-p}
    \right\rfloor+2^sDd^t
    \Bigg(\lfloor Cd^q\rfloor+1+
    \sum_{j=\lfloor Cd^q\rfloor+2}^{\left\lfloor Cd^q({\varepsilon}/{\sqrt2})^{-p}\right\rfloor}
    \left(\frac{Cd^q}{j-1}\right)^\frac{2s}{p(s+1)}\Bigg)^{s+1}
    \left(\frac1\varepsilon\right)^{2s}. $$
Using the formula
$$
    \sum_{i=k+1}^{n}j^{-\beta}\le\int_k^nx^{-\beta}\,\mathrm dx=\left\{
    \begin{array}{rl}\ln n-\ln k, &\quad\beta=1, \\ 
    \frac{n^{1-\beta}-k^{1-\beta}}{1-\beta},&\quad\beta\ne 1, \end{array}\right. \\ 
$$
with $2\le k+1\le n$ and $\beta=\frac{2s}{p(s+1)},$ we finally obtain the desired upper bounds. 
\end{proof}


\begin{remark}\label{remopt}
The algorithm $\Phi_n^d$ is not optimal for $N_n^d$ if information is contaminated by 
noise~\eqref{noisyinf}. Indeed, we have by \eqref{radworst} that 
\begin{eqnarray}\label{rpom}
\rw(N_n^d)&=&\max\bigg\{\bigg(\sum_{i=1}^\infty a_i^2\lambda_{d,i}\bigg)^{1/2}:
\;\sum_{i=1}^na_i^2\le 1,\;|a_i|\le\sigma_i,\,1\le i\le n\bigg\} \\ \nonumber
&=&\sqrt{\sum_{i=1}^\ell\sigma_i^2\lambda_{d,i}+
\bigg(1-\sum_{i=1}^\ell\sigma_i^2\bigg)\lambda_{d,\ell+1}}
\,=\,\sqrt{\sum_{i=1}^\ell\sigma_i^2(\lambda_{d,i}-\lambda_{d,\ell+1})+\lambda_{d,\ell+1}},
\end{eqnarray}
where $\ell$ is the largest $k$ satisfying $\sum_{i=1}^k\sigma_i^2<1,$ cf. \eqref{error1}. Nevertheless, 
$\Phi_n^d$ gives optimal exponents of tractability when one relies only on information $N_n^d.$ 

To see this, let $N_m^d$ be information \eqref{noisyinf} that uses precisions $\sigma_i$ and whose radius 
is at most $\varepsilon\sqrt{\lambda_{d,1}}.$ Then $m\ge n=\nsw(\varepsilon,d).$ A crucial observation 
is that, in view of \eqref{minn} and \eqref{rpom}, we then have $\sum_{i=1}^n\sigma_i^2\le 1.$ Hence 
the cost of $N_m^d$ can be lower bounded by minimization of $n+\sum_{i=1}^n\sigma_i^{-2s}$ (which 
does not exceed $\cw(N_m^d)$) with respect to the condition 
$\sum_{i=1}^n\sigma_i^2\lambda_{d,i}\le\varepsilon^2\lambda_{d,1}$ 
(which is weaker than $\rw(N_m^d)\le\varepsilon\sqrt{\lambda_{d,1}}$). In this way we obtain
$$ \cw_\$(N_m^d)\,\ge\,n+Dd^t\left(\sum_{i=1}^n
  \left(\frac{\lambda_{d,i}}{\lambda_{d,1}}\right)^\frac{s}{s+1}\right)^{s+1}\varepsilon^{-2s}.$$
This bound differs from the upper bound in \eqref{compform} at most by the factor of $2^s,$ which
does not influence the exponents of polynomial tractability.
\end{remark}

We believe that the tractability exponents in (iii) of Theorem \ref{thm:wwpoly} are best possible, 
but a formal justification of this fact is missing. The point is that these exponents are obtained using 
particular information $N_n^d.$ On one hand, Proposition~\ref{prop1} of Appendix shows that 
this information is indeed optimal in some situations, even if we fix precisions $\sigma_i.$ 
On the other hand, the following example shows that the cost can be sometimes significantly reduced 
by applying more sophisticated information.

\begin{example}
Suppose we approximate vectors $\vec x=(x_1,x_2)^T\in\mathbb R^2$ in the $\ell_2$ norm. 
Consider noisy information consisting of $2n$ observations 
$\mathbf y=(\vec y_1^{\,T},\ldots,\vec y_n^{\,T}),$ where
\begin{equation}\label{rotinf}
\mathbb R^2\ni\,\vec y_i=(R_i\vec x)^T+\vec e_i\,,\qquad\|\vec e_i\|_\infty\le\sigma<1,
\end{equation} and
$$R_i=\left(\begin{array}{rr}\cos\theta_i & -\sin\theta_i\\ \sin\theta_i & \cos\theta_i\end{array}\right),
\qquad\theta_i=\frac{\pi(i-1)}{2n},\qquad 1\le i\le n,$$
is the clockwise rotation through the angle $\theta_i.$ Note that $n=1$ corresponds to 
$\mathbf y=\vec x+\vec e,$ $\|\vec e\,\|_\infty\le\sigma,$ which is the information exploited 
in the proof of Theorem \ref{thm:wwpoly} for this particular problem. 

Using geometrical arguments one can easily show that, given $n\ge 1$ and $\varepsilon\in(0,1),$ 
one has to use precision $\sigma_n(\varepsilon)=\varepsilon\cos\big(\frac{\pi}{4n}\big)$ to get an 
$\varepsilon$-approximation of $\vec x,$ and then the cost of the $\varepsilon$-approximation equals
$\mathrm{cost}(n,\varepsilon)=2n\$(\sigma_n(\varepsilon)),$ where $\$$ is a cost function. 
Let $\$(\sigma)=1+\sigma^{-2s}$ with $s>0.$ Then the cost is
\begin{equation}\label{psicost}
\mathrm{cost}_s(n,\varepsilon)=2n\left(1+\big(\varepsilon\cos(\tfrac{\pi}{4n})\big)^{-2s}\right).
\end{equation}
Taking $n^*=\big\lceil\frac\pi 4\sqrt{2s}\,\big\rceil$ we have the asymptotic equality
$$\mathrm{cost}_s(n^*,\varepsilon)\approx\pi\sqrt{\frac{s}{2}}
\left(1+\sqrt{\mathrm e}\,\varepsilon^{-2s}\right)\quad\mbox{as}\quad s\to\infty,$$
where we used the fact that $\lim_{x\to 0}\left(\cos x\right)^{-1/x^2}=\sqrt\mathrm e.$ 
Hence for $0<\varepsilon<1$ we have
$$\frac{\mathrm{cost}_{s}(1,\varepsilon)}{\mathrm{cost}_{s}(n^*,\varepsilon)}
\approx\frac{2\big(1+2^s\varepsilon^{-2s}\big)}{\pi\sqrt{s/2}\,
\big(1+\sqrt{\mathrm e}\,\varepsilon^{-2s}\big)}
\approx\frac{2^{s+1}}{\pi\sqrt{s\,\mathrm e/2\,}}\,.$$
As we can see, for large $s$ the `rotated' information offers a serious improvement compared to 
the `un-rotated' information consisting of only $2$ observations as in the case of exact information. 
\end{example}

\subsection{Weak tractability and the curse of dimensionality}

\begin{theorem}\label{thm:wwweak}
Consider a multivariate problem $\mathcal S=\{S_d\}_{d\ge 1}.$
\begin{itemize}
\item[(i)] Suppose that the problem with noisy information is weakly tractable. Then
\begin{itemize}
\item[$\bullet$] it is weakly tractable for exact information, and
\item[$\bullet$] the cost function grows sub-exponentially in $\sigma^{-1}+d.$
\end{itemize}\smallskip
\item[(ii)] Suppose that 
\begin{itemize}
\item[$\bullet$] the problem is weakly tractable for exact information, and
\item[$\bullet$] the cost function grows polynomially in $\sigma^{-1}$ and sub-exponentially in $d.$
\end{itemize}
Then the same problem with noisy information is weakly tractable.
\smallskip
\item[(iii)] Suppose that
\begin{itemize}
\item[$\bullet$] the problem is strongly polynomially tractable for exact information with $p<2,$ and
\item[$\bullet$] the cost function grows sub-exponentially in $\sigma^{-1} +d.$
\end{itemize}
Then the same problem with noisy information is weakly tractable.
\end{itemize}
\end{theorem}

\begin{proof} 
To show (i) we use Lemma \ref{simplelemma}. 
On one hand we have $\nsw(\varepsilon,d)\le\nw_\$(\varepsilon,d),$ which implies
$$\lim_{\varepsilon^{-1}+d\to\infty\,}\frac{\ln\big(\nsw(\varepsilon,d)\big)}{\varepsilon^{-1}+d}
\le\lim_{\varepsilon^{-1}+d\to\infty\,}
\frac{\ln\big(\nw_\$(\varepsilon,d)\big)-\ln\big(\$(1,d)\big)}{\varepsilon^{-1}+d}=0.$$
On the other hand $\$(\sigma,d)\le\nw_\$(\sigma,d),$ which implies
$$\lim_{\sigma^{-1}+d\to\infty\,}\frac{\ln\big(\$(\sigma,d)\big)}{\sigma^{-1}+d}
\le\lim_{\sigma^{-1}+d\to\infty\,}\frac{\ln\big(\nw_\$(\sigma,d)\big)}{\sigma^{-1}+d}=0.$$

Now we show (ii). Suppose that $\$(\sigma,d)\le 1+D\sigma^{-2s}\kappa(d)$ where 
$\lim_{d\to\infty}\ln\big(\kappa(d)\big)/d=0.$ Proceeding as in the proof of 
(iii) of Theorem \ref{thm:wwpoly} we get from \eqref{compform} that 
$$\nw_\$(\varepsilon,d)\le n+2^sD\,\kappa(d)\,n^{s+1}\varepsilon^{-2s}\quad\mbox{where}\quad 
n=\nsw\big(\varepsilon/\sqrt 2,d\big).$$ 
Hence, if the problem is weakly tractable for exact information, then
\begin{eqnarray*}
\lim_{\varepsilon^{-1}+d\to\infty\,}\frac{\ln\big(\nw_\$(\varepsilon,d)\big)}{\varepsilon^{-1}+d}&=&
\lim_{\varepsilon^{-1}+d\to\infty\,}\frac{\ln\big(\kappa(d)\big)+(s+1)\ln n+
2s\ln(1/\varepsilon)}{\varepsilon^{-1}+d}\\
&=&(s+1)\lim_{\varepsilon^{-1}+d\to\infty\,}
\frac{\ln\big(\nsw(\varepsilon/\sqrt 2,d)\big)}{\varepsilon^{-1}+d}=0,
\end{eqnarray*}
which means that the problem with noisy information is also weakly tractable.

To show (iii), suppose that the problem with noisy information is strongly tractable for 
exact information, $\nsw(\varepsilon,d)\le C\varepsilon^{-p}$ where $p<2.$ Then, by \eqref{lambdasb}, 
$$A:=\sum_{j=1}^\infty\frac{\lambda_{d,j}}{\lambda_{d,1}}\le 
1+C^{2/p}\sum_{j=1}^\infty j^{-2/p}\le 1+\frac{p\,C^{2/p}}{2-p}<+\infty.$$
Let $$n=\nsw\bigg(\frac{\varepsilon}{\sqrt 2},d\bigg)\le C\bigg(\frac{\sqrt 2}{\varepsilon}\bigg)^p.$$ 
For the algorithm $\Phi_n^d$ using noisy information $N_n^d$ with fixed precision 
$\sigma=\frac{\varepsilon}{\sqrt{2A}}$ we have by \eqref{error1} that
$$\ew(S_d,N_n^d,\Phi_n^d)=\sqrt{\sigma^2\sum_{i=1}^n\lambda_{d,i}+\lambda_{d,n+1}}\le
\sqrt{\lambda_{d,1}}\,\sqrt{\sigma^2A+\frac{\lambda_{d,n+1}}{\lambda_{d,1}}}\le
\sqrt{\lambda_{d,1}}\,\sqrt{\frac12\varepsilon^2+\frac12\varepsilon^2}=
\sqrt{\lambda_{d,1}}\,\varepsilon,$$
and $$\nw_\$(\varepsilon,d)\le\cw(N_n^d)=n\,\$\bigg(\frac{\varepsilon}{\sqrt{2A}},d\bigg)\le 
C\bigg(\frac{\sqrt 2}{\varepsilon}\bigg)^p\$\bigg(\frac{\varepsilon}{\sqrt{2A}},d\bigg).$$ Hence 
\begin{eqnarray*}
\lim_{\varepsilon^{-1}+d\to\infty}\frac{\ln\big(\nw_\$(\varepsilon,d)\big)}{\varepsilon^{-1}+d}&\le&
\lim_{\varepsilon^{-1}+d\to\infty}\frac{\ln C+p\big(\tfrac12\ln 2+\ln(\frac1\varepsilon)\big)
+\ln\$\big(\frac{\varepsilon}{\sqrt{2A}},d\big)}{\varepsilon^{-1}+d}\\
&=&\sqrt{2A}\lim_{\varepsilon^{-1}+d\to\infty}\,\frac{\ln\$\big(\frac{\varepsilon}{\sqrt{2A}},d\big)}
{\frac{\sqrt{2A}}{\varepsilon}+d\sqrt{2A}}=0,
\end{eqnarray*}
where we used sub-exponential growth of the cost function.
\end{proof}

\begin{remark}
The sufficient conditions in (iii) of Theorem \ref{thm:wwweak} for weak tractability in the case 
of noisy information can be generalized as follows. Suppose that 
$$\nsw(\varepsilon,d)\le C\varepsilon^{-p}\kappa(d),$$ 
where $\kappa(d)$ grows sub-exponentially in $d.$ Then 
$$A_n^d:=\sum_{j=1}^n\frac{\lambda_{d,j}}{\lambda_{d,1}}
\,\preccurlyeq\,\big(\kappa(d)\big)^{\frac2p}\left\{\begin{array}{ll}1,&\;p<2,\\ 
\ln n,&\;p=2,\\n^{1-2/p},&\;p>2.\end{array}\right.$$
Applying the information $N_n^d$ and algorithm $\Phi_n^d$ with $n=\nsw(\frac{\varepsilon}{\sqrt 2},d)$
and fixed precision $\sigma=\frac{\varepsilon}{\sqrt{A_n^d}},$ as in the proof of 
Theorem~\ref{thm:wwweak}, we obtain that 
$\ew(S_d,N_n^d,\Phi_n^d)\le\sqrt{\lambda_{d,1}}\,\varepsilon$ and 
$$\nw_\$(\varepsilon,d)\le\cw(N_n^d)\preccurlyeq
\varepsilon^{-p}\kappa(d)\,\$\big(\hat\varepsilon,d\big),$$ 
where
$$\hat\varepsilon=\hat\varepsilon(\varepsilon,d)=\left\{\begin{array}{ll}
\varepsilon\,\big(\kappa(d)\big)^{-1/p},&\;p<2,\\
\varepsilon\,\big(\kappa(d)\ln(\kappa(d)\varepsilon^{-2}\big)^{-1/2},&\;p=2,\\
\varepsilon^{p/2}\big(\kappa(d)\big)^{-1/2},&\;p>2.\end{array}\right.$$
Hence the problem is weakly tractable for noisy information if the cost function satisfies
\begin{equation}\label{condweak}
\lim_{\varepsilon^{-1}+d\to\infty}\,\frac{\ln\$\big(\hat\varepsilon,d\big)}{\varepsilon^{-1}+d}=0.
\end{equation}
Observe that (iii) of Theorem \ref{thm:wwweak} is obtained by taking $p<2$ and $\kappa(d)=1,$ 
in which case $\hat\varepsilon=\varepsilon.$

It is not clear whether the condition \eqref{condweak} is not only sufficient, but also necessary 
for weak tractability.
\end{remark}

\medskip
Since intractability is defined as the lack of weak tractability, necessary and sufficient conditions
for a problem to be intractable follow immediately from Theorem \ref{thm:wwweak}.
We move to the curse of dimensionality. 
\begin{theorem}\label{thm:wwcurse}
Consider a multivariate problem $\mathcal S=\{S_d\}_{d\ge 1}.$
\begin{itemize}
\item[(i)] Suppose that 
\begin{itemize}
\item[$\bullet$] the problem with exact information suffers from the curse of dimensionality, or 
\item[$\bullet$] the cost function grows exponentially in $d$ for some $\sigma_0\ge 0.$
\end{itemize}
Then the same problem with noisy information also suffers from the curse.
\smallskip
\item[(ii)] Suppose the problem with noisy information suffers from the curse of dimensionality. 
Then
\begin{itemize}
\item[$\bullet$] the problem with exact information also suffers from the curse, or
\item[$\bullet$] the cost function grows faster than polynomially in $\sigma^{-1},$ 
or grows exponentially in $d.$ 
\end{itemize} 
\smallskip
\item[(iii)] Suppose the problem with noisy information suffers from the curse of dimensionality.
Then
\begin{itemize}
\item[$\bullet$] the problem is not strongly polynomially tractable for exact information, or 
\item[$\bullet$] the cost function grows exponentially in $d$ for some $\sigma_0\ge 0.$ 
\end{itemize}
\end{itemize}
\end{theorem}

\begin{proof}
To show (i) it suffices to use again Lemma \ref{simplelemma}. If the curse is present for exact 
information then, owing to $\nw_\$(\varepsilon_0,d)\ge\nsw(\varepsilon_0,d)\$(1,1),$ it is also present 
for noisy information. If the cost function grows exponentially in $d$ for some 
$\sigma_0>0$ then for $\varepsilon_0=\sigma_0$ we have
$$\limsup_{d\to\infty}\frac{\ln\big(\nw_\$(\varepsilon_0,d)\big)}{d}\ge
\limsup_{d\to\infty}\frac{\ln\big(\$(\sigma_0,d)\big)}{d}>0.$$

To show (ii), assume that there is no curse for exact information, and 
$\$(\sigma,d)\le 1+D\sigma^{-2s}\kappa(d),$ where $\lim_{d\to\infty}\ln(\kappa(d))/d=0.$ 
Then, applying the reasoning from the proof (ii) of Theorem \ref{thm:wwweak} we get that 
$\nw_\$(\varepsilon,d)\le n+2^sD\kappa(d)n^{s+1}\varepsilon^{-2s}$ with 
$n=\nsw(\varepsilon/\sqrt{2},d).$ Hence
$$\lim_{d\to\infty}\frac{\ln\big(\nw_\$(\varepsilon,d)\big)}{d}=
(s+1)\lim_{d\to\infty}\frac{\ln\big(\nsw(\varepsilon/\sqrt{2},d\big)}{d}=0,$$
which means that the problem with noisy information does not suffer from the curse.

And finally, to show (iii) we assume that the problem is strongly polynomially tractably for exact 
information, i.e., $\nsw(\varepsilon,d)\le C\varepsilon^{-p},$ and that $\$(\sigma,d)$ grows 
sub-exponentially in $d$ for all $\sigma>0.$ Using
$\sum_{j=1}^n\lambda_{d,j}/\lambda_{d,1}\le n$ and  
proceeding as in the proof of (iii) of Theorem \ref{thm:wwweak} we obtain 
for $n(\varepsilon)=\big\lceil C\big(\frac{\sqrt 2}{\varepsilon}\big)^p\big\rceil$ that 
$$\lim_{d\to\infty}\frac{\ln\big(\nw_\$(\varepsilon,d)\big)}{d}\le
\lim_{d\to\infty}\frac{\ln\$\big(\frac{\varepsilon}{2n(\varepsilon)},\,d\big)}{d}=0,$$
which means that there is no curse for noisy information.
\end{proof}

\section{Worst case setting with Gaussian noise}\label{sec:random}

In this section we assume that the noise is random.
That is, information about $f$ is given as $\by=(y_1,y_2,\ldots,y_{n(\by)})$ where
$$y_i=L_i(f;y_1,\ldots,y_{i-1})+e_i,\qquad e_i\sim\mathcal N\big(0,\sigma_i^2(y_1,\ldots,y_{i-1})\big),$$
and the noise coming from different observations is independent.  
The (total) cost of information $N=\{\pi_f\}_{f\in F},$ where $\pi_f$ is the probability distribution of 
information $\mathbf y$ for given $f,$ is defined as
$$\ca_\$(N)=\sup_{\|f\|_{F_d}\le 1}\,\int_Y
      \sum_{i=1}^{n(\by)}\$\big(\sigma_i(y_1,\ldots,y_{i-1})\big)\,\pi_f(\mathrm d\mathbf y),$$
and the error of an algorithm $\Phi$ using information $N$ as
$$\ea(S_d,N,\Phi)=\sup_{\|f\|_{F_d}\le 1}\,\bigg(\int_Y\|S_d(f)-\Phi(\by)\|_{G_d}^2\,\pi_f(d\mathbf y)\bigg)^{1/2}.$$
As before, $S_d$ are compact operators.

\smallskip
Define an \emph{auxiliary cost function} $\widehat\$$ 
in such a way that $\widehat\$(\sigma,d)$ is the complexity of approximating a real parameter 
$f\in\mathbb R$ in the current setting using the cost function $\$(\,\cdot\,,d).$ 
(We stress that here $f$ is not restricted to the interval $[-1,1].$ Possible approximations use noisy observations 
of $f$ with adaptively chosen precisions $\sigma_i.$) 

We clearly have $\widehat{\widehat\$}=\widehat\$,$ and $\widehat\$\le\$$ since the approximation 
$\tilde f=f+e,$ where $e\sim\mathcal N(0,\sigma^2),$ gives error $\sigma$ at cost $\$(\sigma,d).$

\begin{lemma}\label{simplelemma1}
For all $\varepsilon\in(0,1)$ and $d\ge 1$ we have
$$\na_\$(\varepsilon,d)\ge\frac{\nsw\big(2\,\varepsilon,d\big)+1}{4}.$$
Also, there is $c\in(0,1)$ such that for all $\varepsilon\in(0,c)$ and $d\ge 1$ we have
$$\na_\$(\varepsilon,d)\ge\frac12\,\widehat\$\bigg(\frac{\varepsilon}{c\sqrt 2},d\bigg).$$
\end{lemma}

\begin{proof}
Let an algorithm $\Phi$ using information $N=\{\pi_f\}_{f\in F}$ be such that 
$\ea(S_d,N,\Phi)\le\varepsilon\sqrt{\lambda_1}.$~Let 
$$n=\sup_{\|f\|_{F_d}\le 1}\int_Y n(\mathbf y)\,\pi_f(\mathrm d\mathbf y).$$
Since $\$\ge 1,$ we have $n\le\ca_\$(N).$
Observe that any deterministic algorithm that uses noisy information can be interpreted as 
a randomized algorithm that uses exact information, where the noise is treated as a random parameter.
Then, by \cite[Theorem 4.42]{NoWo08}, there is deterministic algorithm using exact information 
of cardinality $4n-1$ whose worst case error is at most $2\varepsilon.$ Hence
$$\ca_\$(N,\Phi)\ge n\ge\frac{\nsw\big(2\varepsilon,d\big)+1}{4}.$$
Taking the infimum with respect to all $\Phi$ and $N$ we obtain the desired inequality. 

\smallskip
To show the second inequality, we estimate the complexity of our problem from below by the complexity 
of the same problem, but with error taken over the interval $[-f_{d,1}^*,f_{d,1}^*]$ where, 
as before, $f_{d,1}^*$ is the eigenelement corresponding to the largest eigenvalue of $S_d^*S_d.$
This is equivalent to the one-dimensional problem of approximating a prameter $f\in[-1,1]$ from its noisy 
observations that is analyzed in Appendix of \cite{MoPl20}. The worst case error of the latter can be 
lower bounded by the average error with respect to the two-point probability measure $\mu$ that 
assigns $1/2$ to $\pm1.$ For any adaptive information such average error is not larger than 
$c\min(\sigma,1)$ for some $c>0,$ where $\sigma$ is such that
$$\sigma^{-2}=\int_{-1}^1\int_Y\sigma_{\mathbf y}^{-2}\,\pi_f(\mathrm d\mathbf y)\mu(\mathrm df)
=\int_Y\sigma_{\mathbf y}^{-2}\mu_1(\mathrm d\mathbf y),
\qquad\sigma_{\mathbf y}^{-2}=\sum_{i=1}^{n(\mathbf y)}\sigma_i^{-2}\big(y_1,\ldots,y_{i-1}\big),$$ 
and $\mu_1$ is the a priori distribution of information $\mathbf y$ on $Y,$ cf. \cite[Lemma 3]{MoPl20}. 
Another important property is that for any $\sigma_1,\ldots,\sigma_n$ and $\sigma$ such that 
$\sigma^{-2}=\sum_{i=1}^n\sigma_i^{-2}$ we have 
$$\sum_{i=1}^n\$(\sigma_i,d)\ge\widehat\$(\sigma,d),$$
which follows directly from the definition of $\widehat\$.$ 
 
Let $A\subset Y$ be the set of all $\mathbf y$ such that $\sigma_{\mathbf y}^{-2}\le 2\sigma^{-2}.$ 
Then $\mu_1(A)\ge 1/2.$ Hence, if the error is at most $\varepsilon<c$ then $\sigma\le\varepsilon/c$ and 
the cost is at least
$$\int_A\sum_{i=1}^{n(\mathbf y)}\$\big(\sigma_i(y_1,\ldots,y_{i-1}),d\,\big)\,\mu_1(\mathrm d\mathbf y)
\ge\int_A\sum_{i=1}^{n(\mathbf y)}\widehat\$(\sigma_{\mathbf y},d)\,\mu_1(\mathrm d\mathbf y)
\ge\frac12\,\widehat\$\bigg(\frac{\varepsilon}{c\sqrt 2},d\bigg),$$
as claimed.
\end{proof}

We now show some useful properties of $\widehat\$.$ 
For $d\ge 1,$ define the two functions, $h_1$ and $h_{2,\lambda}.$ 
$$(0,+\infty)\ni x\mapsto h_1(x,d)=\$\bigg(\sqrt{\frac{1}{x}},\,d\bigg),\qquad\mbox{and}$$ 
$$(0,\lambda)\ni x\mapsto h_{2,\lambda}(x,d)=\$\bigg(\sqrt{\frac{\lambda x}{\lambda-x}},\,d\bigg).$$

\begin{lemma}\label{auxcost}
For any $d\ge 1$ we have the following.
\begin{itemize}
\item[(i)] Suppose that $h_1(\,\cdot\,,d)$ is concave, and $h_{2,\lambda}(\,\cdot\,,d)$ is convex
for all $\lambda$ sufficiently large. Then 
$$\widehat\$(\,\cdot\,,d)=\$(\,\cdot\,,d).$$
\item[(ii)] Suppose there is a line $\ell(x)=\alpha x$ supporting $h_1(\,\cdot\,,d)$ at some $x_0>0,$ i.e., 
$\ell(x_0)=h_1(x_0,d)$ and $\ell(x)\le h(x,d)$ for all $x\ge 0.$ Then for all $\sigma>0$ we have 
$$\frac\alpha{\sigma^2}\le
\widehat\$(\sigma,d)\le\bigg\lceil\frac{\sigma_0^2}{\sigma^2}\bigg\rceil\,\frac\alpha{\sigma_0^2},
\qquad\mbox{where}\quad\sigma_0^2=1/x_0.$$
\item[(iii)] We always have
$$\widehat\$(\sigma,d)\le\bigg\lceil\frac{\sigma_0^2}{\sigma^2}\bigg\rceil\,\$(\sigma_0,d),
\qquad\mbox{for any}\quad\sigma_0>0.$$
\end{itemize}
\end{lemma}

\begin{proof}
(i) We already noticed that $\widehat\$\le\$.$ To bound $\widehat\$$ from below we use 
the average case complexity of approximating $f\in\mathbb R$ from observations of $f$ with Gaussian noise, 
where the average squared error and average cost are taken with respect to the one-dimensional Gaussian 
distribution $\mu_\lambda$ with mean zero and variance $\lambda.$  

Consider first nonadaptive information consisting of $n$ observations $y_i=f+e_i$ with precisions $\sigma_i.$ 
Then the optimal algorithm is 
$$\phi^*(\mathbf  y)=\frac{\sigma^2\lambda}{\sigma^2+\lambda}\sum_{i=1}^n\frac{y_i}{\sigma_i^2}\,,
\quad\mbox{where}\quad\sigma^{-2}=\sum_{i=1}^n\sigma_i^{-2},$$
and its average squared error depends only on $\sigma^2$ and $\lambda$ and equals 
$$\int_F\int_{\mathbb R^n}|f-\phi^*(\mathbf y)|^2\pi_f(\mathrm d\mathbf y)\mu_\lambda(\mathrm df)=
\frac{\sigma^2\lambda}{\sigma^2+\lambda}\,,$$ 
cf. \cite[Sect. 3.5]{NICC96}. By concavity of $h_1$ we have 
\begin{eqnarray*}
h_1(\sigma^{-2},d)&=&h_1\bigg(\sum_{i=1}^n\sigma_i^{-2},d\bigg)
\le h_1(0,d)+h_1\bigg(\sum_{i=1}^n\sigma_i^{-2},d\bigg)
\le h_1\bigg(\sum_{i=1}^{n-1}\sigma_i^{-2},d\bigg)+h_1(\sigma_n^{-2},d)\\
&\le& h_1\bigg(\sum_{i=1}^{n-2}\sigma_i^{-2},d\bigg)+h_1(\sigma_{n-1}^{-2},d)+h_1(\sigma_n^{-2},d)
\le\cdots\le\sum_{i=1}^n h_1(\sigma_i^{-2},d),
\end{eqnarray*}
so that $\$(\sigma,d)\le\sum_{i=1}^n\$(\sigma_i,d).$ This means that the cheapest way of obtaining 
an approximation with the average squared error $\sigma^2<\lambda$ using nonadaptive information 
is to use just one observation $y=f+e$ with variance 
$\sigma_1^2=\sigma^2\lambda/(\lambda-\sigma^2),$ for which the cost is
$$\psi_\lambda(\sigma)=\$\bigg(\sqrt{\frac{\sigma^2\lambda}{\lambda-\sigma^2}},\,d\bigg).$$
Since, by assumption, the function $\sigma\mapsto\psi_\lambda(\sqrt\sigma\,)$ is convex for large $\lambda,$ 
we can use \cite[Lemma~3.9.2]{NICC96} to claim, that the cost $\psi_\lambda(\sigma)$ cannot be reduced 
using adaptive information. Hence $\widehat\$(\sigma)\ge\psi_\lambda(\sigma),$ and letting 
$\lambda\to\infty$ we obtain $\widehat\$(\sigma)\ge\$(\sigma),$ which forces $\widehat\$(\sigma)=\$(\sigma).$

\smallskip\noindent
(ii) If the cost function is $\$_1(\sigma,d)=\alpha\sigma^{-2}$ then we have from (i) that $\widehat\$_1=\$_1.$
(Note that here we violate the assumption that the cost is at least $1.$) Since 
$\$(\,\cdot\,,d)\ge\$_1(\,\cdot\,,d)$ then $\widehat\$(\sigma,d)\ge\widehat\$_1(\sigma,d)=\alpha\sigma^{-2}.$
On the other hand, we can approximate $f\in\mathbb R$ with error $\sigma$ using $n$ nonadaptive observations 
with the same precision $\sigma_0/\sqrt{n},$ where $n=\lceil\sigma_0^2/\sigma^2\rceil.$ Hence,
$$\widehat\$(\sigma,d)\le n\,\$(\sigma_0,d)=
\bigg\lceil\frac{\sigma_0^2}{\sigma^2}\bigg\rceil\frac\alpha{\sigma_0^2}.$$

\smallskip\noindent
(iii) The bound can be easily obtained by repetitive observations of variance $\sigma_0^2,$ as in (ii).
\end{proof}

\begin{example}\label{exam}
Assume that the cost function grows polynomially,  
$$\$(\sigma,d)=1+Dd^t\sigma^{-2s}.$$
For $s\le 1$ the function $h_1(\,\cdot\,,d)$ is obviously concave, and 
$$x\mapsto h_{2,\lambda}(x,d)=1+Dd^t\bigg(\frac1x-\frac1\lambda\bigg)^s$$ 
is convex for all $\lambda>0,$ and therefore $\widehat\$=\$.$ 

For $s\ge 1$ the function $\$(\,\cdot\,,d)$ is supported at $x_0=\big((s-1)Dd^t\big)^{-1/s}$ 
by $\ell(x)=\alpha_d x,$ where $$\alpha_d=s(s-1)^{1/s-1}(Dd^t)^{1/s}.$$ 
Hence $\widehat\$(\sigma,d)$ essentially equals $\alpha_d\sigma^{-2}.$
\end{example}

\begin{remark}
Lemma \ref{simplelemma1} is an analogue of Lemma \ref{simplelemma}. In the case of bounded noise 
the corresponding auxiliary cost function would always be $\widehat\$=\$.$ 
\end{remark}

Similarly to the case of bounded noise, for upper bounds on tractability we use 
\begin{equation}\label{approa}
\Phi_n^d(\mathbf y)=\sum_{i=1}^ny_iS_d(f_{d,i}^*),
\end{equation}
where $y_i$ approximates $\langle f,f_{d,i}^*\rangle_{F_d}$ for all $f\in F$ with 
the expected squared error $\sigma_i^2,$ and with cost $\widehat\$(\sigma_i,d).$
Then, for the corresponding information we have
\begin{equation}\label{error2}  \ea(S_d,N^d_n,\Phi^d_n)=
      \sqrt{\sum_{i=1}^n\sigma_i^2\lambda_{d,i}+\lambda_{d,n+1}}\,,
\end{equation}
and $\ca_\$(N_n^d)=\sum_{i=1}^n\widehat\$(\sigma_i,d).$ 
Note that $\ea(S_d,N_n^d,\Phi_n^d)=\ew(S_d,N_n^d,\Phi_n^d),$ where in the deterministic case the noise 
of $i$th observation is bounded by $\sigma_i,$ cf. \eqref{error1}. Hence we can adopt the proof technique
from Section~\ref{sect:bounded} with the cost function $\widehat\$$ to obtain complexity bounds in 
the case of random noise. In particular, Lemma \ref{auxcost}(iii) gives the following general upper bound.

\begin{corollary}\label{corol} 
Suppose that $\nsw(\varepsilon,d)\le Cd^q\varepsilon^{-p}.$ Then for any fixed $\sigma_0>0$ we have
$$\na_\$(\varepsilon,d)\,\preccurlyeq\,\sigma_0^2\,\$(\sigma_0,d)\,d^{\,2q}
\left\{\begin{array}{rl} 
\varepsilon^{-2p}, &\quad p>1,\\ 
\ln^2(1/\varepsilon)\,\varepsilon^{-2},&\quad p=1,\\ 
\varepsilon^{-2}, &\quad p<1.\end{array}\right. $$
\end{corollary}

\subsection{Polynomial tractability}

\begin{theorem}\label{thm:wapoly}
Consider a multivariate problem $\mathcal S=\{S_d\}_{d\ge 1}.$ 
\begin{itemize}
\item[(i)] The problem with noisy information is polynomially tractable if and only if 
\begin{itemize}
\item[$\bullet$] it is polynomially tractable for exact information, and
\item[$\bullet$] the auxiliary cost function grows polynomially in $d$ for some $\sigma_0>0.$
\end{itemize} \smallskip
\item[(ii)] The problem with noisy information is strongly polynomially tractable if and only if
\begin{itemize}
\item[$\bullet$] it is strongly polynomially tractable for exact information, and
\item[$\bullet$] the auxiliary cost function is bounded in $d$ for some $\sigma_0>0.$ 
\end{itemize} \smallskip
\item[(iii)]
Suppose that $\nsw(\varepsilon,d)\le Cd^q\varepsilon^{-p}$ and 
$\$(\sigma,d)\le 1+Dd^t\sigma^{-2s}.$ \\ 
 If $p=0$ and $s=0$ then 
$\nw_\$(\varepsilon,d)\preccurlyeq d^{t+q};$ otherwise
$$  \nw_\$(\varepsilon,d)\,\preccurlyeq\, d^{\overline t+q(\overline s+1)}
             \left\{\begin{array}{rl} 
               \varepsilon^{-p(\overline s+1)}, &\; p(\overline s+1)>2\overline s, \\
               \ln^{\overline s+1}(1/\varepsilon)\,\varepsilon^{-2\overline s}, 
               &\; p(\overline s+1)=2\overline s,\\
               \varepsilon^{-2\overline s}, &\; p(\overline s+1)<2s,
           \end{array}\right.
$$
where $\overline t=\min(t,t/s)$ and $\overline s=\min(s,1).$
\end{itemize}
\end{theorem}

\begin{proof}
Suppose the problem with noise is polynomially tractable, i.e., 
$\na(\varepsilon,d)\le Cd^q\varepsilon^{-p}$ for $\varepsilon\in(0,1)$ and $d\ge 1.$
Then we have by Lemma \ref{simplelemma1} that, on one hand, 
$$\nsw(\varepsilon,d)\le 4\,\na\big(\tfrac\varepsilon 2,d\big)-1\le 
4\,Cd^q\big(\tfrac\varepsilon 2\big)^{-p}\preccurlyeq d^q\varepsilon^{-p},$$
and, on the other hand, 
$$\widehat\$(\sigma_0,d)\le 2\,\na\big(c\sqrt 2\,\sigma_0,d\big)
\le 2+2\,Cd^q\big(c\sqrt 2\,\sigma_0\big)^{-p}.$$
This proves the necessary conditions in (i) and (ii). 
The sufficient conditions follow from Corollary~\ref{corol}.
To show (iii) it suffices to note that 
$\widehat\$(\sigma,d)\preccurlyeq d^{\,\overline t}\sigma^{\,\overline s},$
as in Example \ref{exam}, and use Theorem \ref{thm:wwpoly}. 
\end{proof}

\begin{remark}
As in the case of bounded noise, see Remark \ref{remopt}, for Gaussian noise the algorithm 
$\Phi_n^d$ is not optimal for information $N_n^d;$ however, one can show that 
the exponents of tractability in (iii) of Theorem \ref{thm:wapoly} cannot be improved when one relies only on information $N_n^d.$ If arbitrary information is allowed then the optimal exponents  are unknown.
The point is that we do not know in general how to optimally choose the information. 
For a particular case, see Proposition \ref{prop:appdx1} of Appendix. 
\end{remark}

\subsection{Weak tractability and the curse of dimensionality}

\begin{theorem}\label{thm:waothers}
Consider a multivariate problem $\mathcal S=\{S_d\}_{d\ge 1}.$ 
\begin{itemize}
\item[(i)] The problem with noisy information is weakly tractable if and only if 
the same problem with exact information is weakly tractable.
\item[(ii)] The problem with noisy information suffers from the curse of dimensionality 
if and only if the same problem with exact information suffers from the curse.
\end{itemize}
\end{theorem}

\begin{proof}
(i) The necessary condition follows from the first inequality in Lemma \ref{simplelemma1}. 
For sufficiency it is enough to consider the information $N_n^d$ and algorithm $\Phi_n^d$ 
with $n=\nsw(\varepsilon/\sqrt 2,d)$ and $\sigma_i^2=\varepsilon^2/(2n)$ for all $i.$ Then
the error is $\varepsilon$ and the cost is 
$n\,\widehat\$(\varepsilon/\sqrt n,d)\le\beta n^2/\varepsilon^2$  for some $\beta>0.$ 
This implies
$$\lim_{\varepsilon^{-1}+d\to\infty}\frac{\ln\big(\na_\$(\varepsilon,d)\big)}{\varepsilon^{-1}+d}
\le\lim_{\varepsilon^{-1}+d\to\infty}\frac{\ln\beta+2\ln n+2\ln(1/\varepsilon)}{\varepsilon^{-1}+d}
=0.$$

\smallskip\noindent
(ii) The sufficiency follows again from the first inequality in Lemma \ref{simplelemma1}. 
For necessity it is enough to consider the same approximation as in (i) to get that if the problem 
with exact information does not suffer from the curse then for all $\varepsilon$ we have
$$\lim_{d\to\infty}\frac{\ln\big(\na_\$(\varepsilon,d)\big)}{d}
\le\lim_{d\to\infty}\frac{\ln\beta+2\ln n+2\ln(1/\varepsilon)}{d}=0,$$
i.e., the problem with noise as well does not suffer from the curse.
\end{proof}

\section*{Acknowledgments}
\noindent
The research of the authors was supported by the National Science Centre, Poland, under project 2017/25/B/ST1/00945.

\section*{Appendix}

Let $F$ and $G$ be unitary spaces with $\dim F=m<+\infty,$ and $S:F\to G$ a nonsingular 
linear operator. Let $\{x_i^*\}_{i=1}^m$ be an orthonormal basis in $F$ of eigenelements 
of the operator $S^*S$ and $\{\lambda_i\}_{i=1}^m$ the corresponding eigenvalues, where 
$\lambda_1\ge\lambda_2\ge\cdots\ge\lambda_m>0.$ 

\smallskip
We fix the numbers
$$0=\sigma_1=\cdots=\sigma_{n_0}<\sigma_{n_0+1}\le\cdots\le\sigma_n$$ 
(where $n_0=0$ if all $\sigma_i$s are positive) and for any functionals $L_i$ define 
$${\mathrm R}^{\mathrm{ww}}(L_1,L_2,\ldots,L_n)=
\max\big\{\|Sx\|_Y:\;x\in X,\,|L_ix|\le\sigma_i,\,1\le i\le n\big\}.$$
This is obviously the radius of information $y_i=L_ix+e_i,$ $1\le i\le n,$ in the worst case 
setting with bounded noise, $|e_i|\le\sigma_i,$ where the worst case error is taken
with respect to the whole space $F$ (instead of the unit ball of $F$), cf. \eqref{radworst}.

\begin{proposition}\label{prop1} Let $n=m.$ 
For any functionals $L_i$ with $\|L_i\|\le 1$ for $1\le i\le n,$ we have
\begin{equation}\label{todo}
{\mathrm R}^{\mathrm{ww}}(L_1,L_2,\ldots,L_n)\ge
\bigg(\sum_{i=n_0+1}^n\sigma_i^2\lambda_i\bigg)^{1/2}.
\end{equation}
The equality above holds for $L_i^*=\langle\,\cdot\,,x_i^*\rangle_F.$
\end{proposition}

\begin{proof}
We can assume without loss of generality that the functionals $L_i$ are linearly independent 
since otherwise ${\mathrm R}^{\mathrm{ww}}(L_1,L_2,\ldots,L_n)$ is infinite. 

Let $\{x_j\}_{j=1}^n$ be the basis in $F$ that is adjoint to $\{L_i\}_{i=1}^n,$ i.e., 
$L_ix_j=\delta_{i,j}$ for $1\le i,j\le n.$ Let 
$$V_i=\mathrm{span}\big(x_j:\,1\le j\le n,\,j\ne i\big).$$ Observe that
$$\|L_i\|=\max_{x\ne 0}\frac{|L_ix|}{\|x\|_F}=\max_{v\in V_i}\frac{1}{\|x_i+v\|_F}
=\frac{1}{\mathrm{dist}(x_i,V_i)}.$$ 
Hence the condition $\|L_i\|\le 1$ is equivalent to $$\mathrm{dist}(x_i,V_i)\ge 1.$$ 

Since for any $x=\sum_{j=1}^nc_jx_j\in F$ is $L_ix=c_i,$ inequality \eqref{todo} can be 
equivalently rewritten as 
\begin{equation}\label{modif}
\max\Big\{\,\Big\|\sum_{i=1}^nc_iSx_i\Big\|_G^2:\;|c_i|\le\sigma_i,\,1\le i\le n\Big\}\ge
\sum_{i=1}^n\sigma_i^2\lambda_i.
\end{equation}
To show \eqref{modif} we use induction on $n.$ For $n=1$ we have
$$\max_{|c_1|\le\sigma_1}\|c_1Sx_1\|_G^2=
\max_{|c_1|\le\sigma_1}c_1^2\lambda_1=\sigma_1^2\lambda_1.$$ 
Suppose $n\ge 2.$ Let 
$$W=\mathrm{span}\big(Sx_1,\ldots,Sx_{n-1}\big)$$
and $Z$ be the subspace of $\mathrm{im}\,A$ that is perpendicular to $W.$  
We obviously have $\dim W=n-1$ and $\dim Z=1.$ Let us decompose $S$ as
$$S=P_WS+P_ZS,$$
where $P_W$ and $P_Z$ are the orthogonal projections onto 
$W$ and $Z,$ respectively, and denote by 
$$S_W:V_n\to G\qquad\mbox{and}\qquad S_Z:V_n^\perp\to G$$
the operators $P_WS$ and $P_ZS$ restricted, respectively, 
to $V_n$ and $V_n^\perp.$ Let 
$$\widehat\lambda_1\ge\cdots\ge\widehat\lambda_{n-1}>0$$
be the eigenvalues of 
$S_W^*S_W:V_n\to V_n,$ and $\widehat\lambda_n>0$ be 
the only eigenvalue of $S_Z^*S_Z:V_n^\perp\to V_n^\perp.$ 
Observe that for $1\le i\le n-1$ the distance of $x_i$ from 
the subspace 
$\mathrm{span}\big(x_j:\,j\in\{1,\ldots,n\}\setminus\{i,n\}\big)$ is at least $1.$ 
Then, by induction hypothesis, there are $c_i$ 
such that $|c_i|\le\sigma_i,$ $1\le i\le n-1,$ and
$$\Big\|\sum_{i=1}^{n-1}c_iSx_i\Big\|_G^2\ge
\sum_{i=1}^{n-1}\sigma_i^2\widehat\lambda_i.$$
Letting $v$ be the orthogonal projection of $x_n$ onto $V^\perp$ 
we also have that 
$$\|Sx_n\|_G^2\ge\|P_ZSx_n\|_G^2=\|P_ZSv\|_G^2=\sigma_n^2\widehat\lambda_n.$$
Setting $c_n=\sigma_n$ if 
$\big\langle\sum_{i=1}^{n-1}c_iSx_i,Sx_n\big\rangle_G\ge 0,$ and $c_n=-\sigma_n$ 
otherwise, we obtain
$$\Big\|\sum_{i=1}^{n-1}c_iSx_i+c_nSx_n\Big\|_G^2\ge
\Big\|\sum_{i=1}^{n-1}c_iSx_i\Big\|_G^2+
\big\|Sx_n\big\|_G^2\ge\sum_{i=1}^n\sigma_i^2\widehat\lambda_i.$$
To complete the proof of \eqref{modif} it suffices to use the fact that 
$$\sum_{i=1}^k\widehat\lambda_i\le\sum_{i=1}^k\lambda_i\quad
\mbox{for}\quad 1\le k\le n-1,\quad\mbox{and}\quad 
\sum_{i=1}^n\widehat\lambda_i=\sum_{i=1}^n\lambda_i,$$
which implies 
$\sum_{i=1}^n\sigma_i^2\widehat\lambda_i\ge\sum_{i=1}^n\sigma_i^2\lambda_i.$

\smallskip
The remaining part of the proposition is obvious, since orthogonality of $\{Sx_i^*\}_{i=1}^n$ 
implies that for any $c_i$ with $|c_i|\le\sigma_i$ we have
$$\Big\|\sum_{i=1}^nc_iSx_i^*\Big\|_G^2=\sum_{i=1}^n c_i^2\|Sx_i^*\|_G^2
=\sum_{i=1}^n c_i^2\lambda_i\le\sum_{i=1}^n\sigma_i^2\lambda_i.$$
\end{proof}

Proposition \ref{prop1} should be confronted with the related result in the case of Gaussian noise. Let 
$${\mathrm R}^{\mathrm{wa}}(L_1,L_2,\ldots,L_n)=\inf_\Phi\,\sup_{x\in F}
\bigg(\int_{\mathbb R^n}\|Sx-\Phi(\mathbf y)\|_G^2\,\pi_x(\mathrm d\mathbf y)\bigg)^{1/2},$$
where $\pi_f$ is the $n$-dimensional Gaussian measure with mean $(L_1x,\ldots,L_nx)$ and 
covariance matrix $\mathrm{diag}(\sigma_1^2,\ldots,\sigma_n^2),$ be the mimimal error 
in the worst case setting with Gaussian noise, where again the worst case error is taken with 
respect to the whole space $F.$  In this case the optimal choice of the functionals $L_i$ 
is known for all $n\ge m,$ but their construction is much more complicated.  

\begin{proposition}\label{prop:appdx1}
Let $n\ge m.$ For any functionals $L_i$ with $\|L_i\|\le 1$ for $1\le i\le n,$ we have
\begin{equation}\label{done}
{\mathrm R}^{\mathrm{wa}}(L_1,L_2,\ldots,L_n)\ge
\bigg(\sum_{i=n_0+1}^m\widehat\sigma_i^2\lambda_i\bigg)^{1/2},
\end{equation}
where $\widehat\sigma_{n_0+1}\le\cdots\le\widehat\sigma_m$ minimize the sum
$\sum_{i=n_0+1}^m\eta_i^2\lambda_i$ with respect to all $\eta_{n_0+1}\le\cdots\le\eta_m$
satisfying 
$$\sum_{i=k}^m\eta_i^{-2}\le\sum_{i=k}^n\sigma_i^{-2}\quad\mbox{for}\quad k=n_0+1\le k\le n,$$
and $\sum_{i=n_0+1}^m\eta_i^{-2}=\sum_{i=i_0+1}^n\sigma_i^{-2}.$
The equality in \eqref{done} holds for $L_i^*=\langle\,\cdot\,,x_i^*\rangle_F,$ $1\le i\le n_0,$ and
$$L_{n_0+i}^*=\sigma_{n_0+i}\sum_{j=1}^{m-n_0}w_{i,j}\langle\,\cdot\,,x_{n_0+j}^*\rangle_F,
\qquad 1\le i\le n-n_0,$$
where the matrix $W=\{w_{i,j}\}\in\mathbb R^{(n-n_0)\times(m-n_0)}$ satisfies 
$$W^TW=\mathrm{diag}\big(\widehat\sigma_{n_0+1}^{-2},\ldots,\widehat\sigma_m^{-2}\big)
\quad\mbox{and}\quad\big\|W^T\vec e_i\big\|_2^2=\sigma_{n_0+i}^{-2},\quad 1\le i\le n-n_0.$$
\end{proposition}

\medskip
A proof of this proposition as well as construction of the matrix $W$ can be derived from 
Theorems~3.8.1 and 3.8.2 and Lemma~2.8.1 in \cite{NICC96}. One considers the average case error with 
respect to the Gaussian measure on $\mathbb R^m$ with mean zero and covariance matrix 
$\lambda I_m$ (where $I_m$ is the $m\times m$ identity matrx). The worst case error is obtained as the limiting case when $\lambda$ increases to infinity. 

We add that for $n=m$ we have $L_i^*=\langle\,\cdot\,,x_i^*\rangle_F$ for all $1\le i\le n$ if and only if 
$$\sigma_{n_0+i}^{-2}=\frac{\lambda_{n_0+i}^{1/2}}{\sum_{j=n_0+1}^n\lambda_j^{1/2}}
\sum_{j=n_0+1}^n\sigma_j^{-2}\quad\mbox{for}\quad 1\le i\le n-n_0,$$
and then $\mathrm R\big(\langle\,\cdot,x_1^*\rangle_F,\ldots,\langle\,\cdot,x_n^*\rangle_F\big)
=\left(\sum_{i=n_0+1}^n\sigma_i^2\lambda_i\right)^{1/2}.$

\bigskip


\begin{thebibliography}{99}

\bibitem{Don94}
 Donoho D.L., Statistical estimation and optimal recovery. 
 {\em Annals of Statistics} {\bf 22} (1994), pp. 238-270.
 
\bibitem{MoPl20}
 Morkisz P., Plaskota L., Complexity of approximating H\"older classes from information with 
 varying Gaussian noise. {\em Journal of Complexity} {\bf 60} (2020) 101497.

\bibitem{NoWo08}
  Novak E., Wo\'zniakowski H., {\em Tractability of Multivariate Problems. Volume I: Linear Information}.
  Tracts in Mathematics 6, EMS 2008. 

\bibitem{NoWo10}
  Novak E., Wo\'zniakowski H., 
  {\em Tractability of Multivariate Problems. Volume II: Standard Information for Functionals}.
  Tracts in Mathematics 12, EMS 2010. 

\bibitem{NoWo12}
  Novak E., Wo\'zniakowski H., 
  {\em Tractability of Multivariate Problems. Volume III: Standard Information for Operators}.
  Tracts in Mathematics 18, EMS 2012. 

\bibitem{NICC96}
  Plaskota L., {\em Noisy Information and Computational Complexity}. 
   Cambridge University Press, Cambridge, 1996. 

\bibitem{Pla96a}
  Plaskota L., How to benefit from noise. {\em Journal of Complexity} {\bf 12} (1996), pp. 175-184.

\bibitem{Pla96b}
  Plaskota L., Worst case complexity of problems with random information noise. 
  {\em Journal of Complexity} {\bf 12} (1996), pp. 416--439. 

\end{thebibliography}
\end{document}